\documentclass{article}
\title{Smoothness criteria for Navier-Stokes equations in terms of regularity along the stream lines}
\author{Chi Hin Chan}
\date{}

\usepackage{amsmath}
\usepackage{amssymb}
\usepackage{amsthm}
\usepackage{amsfonts}
\usepackage[all]{xy}
\newtheoremstyle{modthm}
     {15pt}
     {3ptpt}
     {\itshape}
     {}
     {\bfseries}
     {.}
     {.5em}
     {}

\newtheoremstyle{modrem}
     {15pt}
     {0pt}
     {\rmfamily}
     {}
     {\itshape}
     {.}
     {.5em}
     {}

\theoremstyle{modthm}
\newtheorem{thm}{Theorem}
\newtheorem{prop}{Proposition}[section]
\newtheorem{lem}{Lemma}[section]

\theoremstyle{modrem}

\newcommand{\Div}{\mathrm{div}}
\newcommand{\R}{\mathbb{R}}
\begin{document}

\bibliographystyle{plain}
\maketitle
\begin{center}
Department of Mathematics \\
University of Texas at Austin
\end{center}
{ \small \noindent{\bf Abstract:} This article is devoted to a regularity criteria for solutions of the Navier-Stokes equations in terms of regularity along the stream lines. More precisely, we prove that if $u$ is a suitable weak solution for the Navier-Stokes equation on $[0,T]\times \mathbb{R}^{3}$ satisfying the condition that $\frac{|u\cdot \nabla F|}{|u|^{\gamma }} \leqslant A |F|$, in which 
$F= div(\frac{u}{|u|})$, $A$ is some given constant, and $\gamma$ is some positive number with $0 < \gamma < \frac{1}{3}$, then it follows that $u$ is smooth over $(0,T)\times \mathbb{R}^{3}$. }

\vskip0.2cm \noindent {\bf keywords:} Navier-Stokes, regularity
criterion, a priori estimates


\section{Introduction}

In this article, we consider  the Navier-Stokes equation on
$\mathbb{R}^3$, given by
\begin{gather}\label{NS}
\partial_{t} u -\triangle u + \Div (u\otimes u) + \nabla p = 0 , \\
\Div (u) = 0, \label{incompressibility}
\end{gather}
where $u$ is a vector-valued function representing the velocity of
the fluid, and $p$ is the pressure. Note that the pressure depends
in a non local way on the velocity $u$. It can be seen as a Lagrange
multiplier associated to the incompressible condition
(\ref{incompressibility}). The  initial  value problem of the above
equation is endowed with the condition that $u(0, \cdot ) = u_{0}
\in L^2(\mathbb{R}^3)$. Leray~\cite{Leray} and Hopf~\cite{Hopf} had
already established the existence of global weak solutions for the
Navier-Stokes equation. In particular, Leray introduced a notion of
weak solutions for the Navier-Stokes equation, and  proved that, for
every given initial datum $u_{0} \in L^2(\mathbb{R}^3)$, there
exists a global weak solution $u \in L^{\infty}(0, \infty ;
L^2(\mathbb{R}^3)) \cap L^2(0, \infty ; \dot{H}^1(\mathbb{R}^3))$
verifying the Navier-Stokes equation in the sense of distribution.
From that time on, much effort has been devoted to establish the
global existence and uniqueness of smooth solutions to the
Navier-Stokes equation. Different Criteria for regularity of the
weak solutions have been proposed. The Prodi-Serrin conditions (see
Serrin \cite{Serrin}, Prodi \cite{Prodi}, and \cite{Struwe}) states
that any weak Leray-Hopf solution verifying $u\in
L^p(0,\infty;L^q(\R^3))$ with $2/p+3/q=1$, $2\leq p<\infty$, is
regular on $(0,\infty)\times\R^3$. The limit case of
$L^\infty(0,\infty; L^3(\R^3))$ has been solved very recently by L.
Escauriaza, G. Seregin, and V. Sverak (see \cite{Esca}). Here, we just mention a piece of work ~\cite{Chan} by Chi Hin Chan and Alexis Vasseur which is devoted to a log improvement of the Prodi-Serrin criteria in the case in which $p=q=5$. 
Other criterions have been later introduced, dealing with some derivatives
of the velocity. Beale, Kato and Majda \cite{Beale} showed the global
regularity under  the condition that the vorticity
$\omega=\mathrm{curl}\ u$ lies in $L^\infty(0,\infty;L^1(\R^3))$
(see Kozono and Taniuchi for improvement of this result
\cite{Kozono}). Beir\~ao da Veiga show in \cite{Veiga} that the
boundedness of $\nabla u$ in $L^p(0,\infty; L^q(\R^3))$ for
$2/p+3/q=2$, $1<p<\infty$ ensures the global regularity. In
\cite{Constantin}, Constantin and Fefferman gave a condition
involving only the direction of the vorticity. Until more recently, in a short paper~\cite{Vasseur2}, A. Vasseur gave another regularity criteria which states that any Leray-Hopf weak solution $u$ for the Navier-Stokes equation satisfying $div(\frac{u}{|u|}) \in L^{p}(0,\infty ;L^{q}(\mathbb{R}^{3}))$ with $\frac{2}{p}+ \frac{3}{q} \leqslant \frac{1}{2}$ is necessary smooth on $(0, \infty )\times \mathbb{R}^{3}$.
As we can see, the regularity criteria given in ~\cite{Vasseur2} is the one with some integrable condition imposed on $div(\frac{u}{|u|})$. However, the goal of this paper is to obtain the full regularity of a suitable weak solution $u$ under some suitable assumption about the smoothness of $div(\frac{u}{|u|})$ along the steam lines of the fluid. More precisely, the goal of this paper is to prove the following theorem

\begin{thm}\label{goal}
Let $u$ be a suitable weak solution for the Navier-Stokes equation on 
$(0,T]\times\mathbb{R}^{3}$ which satisfies the condition that 
$|\frac{u\cdot \nabla F }{|u|^{\gamma }}|\leqslant A|F|$, in which $A$ is some positive constant, and $\gamma $ is some positive constant for which
 $0< \gamma < \frac{1}{3}$. Then, it follows that $u$ is a smooth solution on 
$(0,T]\times \mathbb{R}^{3}$.
\end{thm}

As for Theorem~\ref{goal}, we note that $F = div(\frac{u}{|u|})$ can be rewritten as $F = -\frac{u\cdot \nabla |u|}{|u|^{2}}$, and hence is the derivative of $|u|$ along the streamlines of the fluid. Then, the condition appearing in the hypothesis of Theorem~\ref{goal} can be seen as a constraint on the second 
derivative along the streamlines. Theorem~\ref{goal} itself shows that such a constraint on the second derivative along the streamlines is enough to give the full regularity of the solution.\\

Before we proceed any further, let us say something about the term suitable weak solution. The concept of suitable weak solutions for Navier-Stokes equations was first introduced by Caffarelli, Kohn, and Nirenberg in ~\cite{Caff} for the purpose of developing the partial regularity theory for solutions of Navier-Stokes equations. By a suitable weak solution for the Navier-Stokes equations, we mean a Leray-Hopf weak solution 
$u\in L^{\infty}(0,T; L^{2}(\mathbb{R}^{3}))\cap L^{2}(0,T; \dot{H}^{1}(\mathbb{R}^{3}))$ which satisfies the following inequality in the sense of distribution on $(0,T)\times \mathbb{R}^{3}$.

\begin{equation*}
\partial_{t} (\frac{|u|^{2}}{2}) + div(\frac{|u|^{2}}{2} u) + div(P u) + |\nabla u|^{2} -\triangle (\frac{|u|^{2}}{2}) \leqslant 0 .
\end{equation*}

 Here, we decide to work with suitable weak solutions instead of just Leray-Hopf weak solutions because suitable weak solutions enjoy some very nice properties such as the partial regularity Theorem due to Caffarelli, Kohn, and Nirenberg in their joint work~\cite{Caff}. Now, let us turn our attention back to Theorem~\ref{goal}.
Indeed the conclusion for Theorem~\ref{goal} will follow at once provided if we can prove the following proposition.

\begin{prop}\label{bounded}
Let $u$ be a suitable weak solution for the Navier-Stokes equation on $(0,1]\times
\mathbb{R}^{3}$ which satisfies the condition that $|\frac{u\cdot \nabla F}{|u|^{\gamma }}|\leqslant A|F|$, where $A$ is some positive constant, and $\gamma $ is some positive number satisfying $0< \gamma < \frac{1}{3}$. It then follows that $u$ is essentially bounded over the region $[\frac{3}{4} , 1]\times \mathbb{R}^{3}$. That is, we have $\|u\|_{L^{\infty}([\frac{3}{4}, 1]\times\mathbb{R}^{3})} < \infty $.
\end{prop}

Before we devote our effort to prove proposition~\ref{bounded}, let us first explain why proposition~\ref{bounded} will lead to the conclusion of Theorem~\ref{goal} as follows. Assume that proposition~\ref{bounded} is indeed true. Without the loss of generality, let us assume that $u$ is a suitable weak solution for the Navier-Stokes equation on $(0,1]\times \mathbb{R}^{3}$  satisfying the hypothesis of Theorem~\ref{goal} (we note that if our suitable weak solution $u$ is over $(0,T]\times \mathbb{R}^{3}$, with $T$ to be some positive number other than $1$, we can always rescale our weak solution $u$). Now, proposition~\ref{bounded} automatically tells us that $u$ is essentially bounded on the region $[\frac{3}{4}, 1]\times \mathbb{R}^{3}$. So, over such a region, we can apply the Serrin criterion with $p=q= \infty $ to conclude that $u$ is smooth over $(\frac{3}{4} ,1)\times \mathbb{R}^{3}$. So, the only question remains is how to justify that $u$ is also smooth over $(0, \frac{7}{8})\times \mathbb{R}^{3}$. So, to finish our job, let $\tau \in (0, \frac{7}{8})$ be arbritary chosen and fixed, and let us consider the function 
$u_{\lambda }(t,x) = \lambda u(\lambda^{2}t, \lambda x )$, with 
$\lambda = (\frac{8\tau }{7})^{\frac{1}{2}}$. Notice that $u_{\lambda}$ is then another suitable weak solution on $(0, 1]\times \mathbb{R}^{3}$, which satisfies the same hypothesis of Theorem~\ref{goal}(with a different constant $A_{\lambda}$, of course). So, we can invoke proposition~\ref{bounded} again to conclude that $u_{\lambda }$ is essentially bounded over 
$[\frac{3}{4},1]\times \mathbb{R}^{3}$. However, this means the same thing as saying that our original suitable weak solution $u$ is essentially bounded over the region 
$[\frac{6\tau }{7} , \frac{8\tau }{7}]\times \mathbb{R}^{3}$, and hence $u$ must be smooth over the region $(\frac{6\tau }{7} , \frac{8\tau }{7})\times \mathbb{R}^{3}$. Since the number $\tau\in (0, \frac{7}{8})$ is arbritary chosen in the above argument, we conclude that $u$
must be smooth over $(0,1)\times \mathbb{R}^{3}$, provided that proposition~\ref{bounded} is valid. So, it is clear that the main task of the whole paper is to prove proposition~\ref{bounded}, which is what we will do in the following sections.

\section{Basic setting of the whole paper}

In order to prove proposition~\ref{bounded}, we would like to use the method of energy decompositions with respect to a sequence of cutting functions 
$v_{k} = \{|u|-R(1-\frac{1}{2^{k}})\}_{+} $ as introduced by A. Vasseur in ~\cite{Vasseur}. Indeed, A. Vasseur was the first to use such a method of energy decompositions inherted from De Giorgi ~\cite{De} to give a proof of the famous Partial Regularity Theorem of Caffarelli, Kohn and Nirenberg (see \cite{Vasseur}). So, we would like to introduce some notation first. Then, we will state one lemma and one proposition which are related to the proof
of proposition~\ref{bounded}. So, let us fix our notation as follow.\\

\begin{itemize}
\item for each $k \geqslant 0$, let $Q_{k} = [T_{k} , 1]\times \mathbb{R}^3$, in which
$T_{k} = \frac{3}{4} - \frac{1}{4^{k+1}}$.

\item for each $k \geqslant 0$, let $v_{k} = \{ |u| -R(1 - \frac{1}{2^k}) \}_{+}$.
\item for each $k \geqslant 0$, let $d_{k} = \frac{R( 1 - \frac{1}{2^k})}{|u|}
\chi_{\{|u| \geqslant R (1- \frac{1}{2^k})\}} |\nabla |u||^{2} + \frac{v_{k}}{|u|} |\nabla
u|^2$.
\item for each $k \geqslant 0$, let $U_{k} = \frac{1}{2}\|v_{k}\|^{2}_{L^{\infty}(T_{k} , 1 ;
L^{2}(\mathbb{R}^3))} + \int_{T_{k}}^{1}\int_{\mathbb{R}^{3}} d_{k}^2 dx\,dt$.
\end{itemize}
With the above setting, we are now ready to state the lemma and
proposition which are related to proposition~\ref{bounded} as follow.

\begin{prop}\label{index}
Let $u$ be a suitable weak solution for the Navier-Stokes equation on 
$[0,1]\times \mathbb{R}^3$ which satisfies the condition that
$|\frac{u\cdot \nabla F}{|u|^{\gamma }}|\leqslant A |F|$, where $A$ is some finite-positive constant, and $\gamma $ is some psoitive number satisfying 
$0 < \gamma < \frac{1}{3}$. Then, there exists some constant $C_{p,\beta }$, depending only on $1<p<\frac{5}{4}$, and $\beta > \frac{6-3p}{10-8p}$,and also some constants $0 < \alpha , K <\infty$, which do depend on our suitable weak solution $u$, such that the following inequality holds

\begin{equation*}
\begin{split}
U_{k}\leqslant \\
& C_{p,\beta }2^{\frac{10k}{3}}
\{\frac{1}{R^{\beta \frac{10-8p}{3p}-\frac{2-p}{p}}} \|u\|_{L^{\infty}(0,1;L^{2}(\mathbb{R}^{3}))}^{2(1-\frac{1}{p})}U_{k-1}^{\frac{5-p}{3p}} +\\
&(1+A)(1+\frac{1}{\alpha })
(1+K^{1-\frac{1}{p}})(1+\|u\|_{L^{\infty}(0,1;L^{2}(\mathbb{R}^{3}))})\times \\
&[ (\frac{1}{R^{\frac{10}{3}-2p\beta +1-\gamma -p }})^{\frac{1}{p}}U_{k-1}^{\frac{5}{3p}}
 + \frac{1}{R^{\frac{10}{3}-2\beta -\gamma }}U_{k-1}^{\frac{5}{3}}]\} ,
\end{split}
\end{equation*}

for every sufficiently large $R>1$.

\end{prop}

Here, let us make some important comments on the conclusion of proposition~\ref{index}. As indicated by the inequality which appears in the conclusion of proposition~\ref{index}, it is important for us to emphasis that those terms such as 
$R^{\beta \frac{10-8p}{3p} - \frac{2-p}{p}}$, $R^{\frac{10}{3}-2p\beta +1-\gamma -p}$, and $R^{\frac{10}{3}-2\beta -\gamma }$ should all appear in the denomerator. But unfortunately, the standard approach of carrying out decompositions on both the energy and pressure by using the same sequence of cutting functions 
$v_{k} = \{|u|-R(1-\frac{1}{2^{k}})\}_{+}$ is not powerful enough to ensure such a result as promised by proposition~\ref{index}. So, in proving proposition~\ref{index}, we will carry out the decomposition of the pressure $P$ by introducing another sequence of cutting functions $w_{k} = \{|u|- R^{\beta} (1-\frac{1}{2^{k}})\}_{+}$, for $k\geqslant 1$, where $\beta > \frac{3}{2}$ should be some suitable index sufficiently close to $\frac{3}{2}$ (for more detail, see inequalities 
(\ref{P1}), (\ref{P2}), and (\ref{P3}) ). We remark that the inequality $\|\chi_{\{ w_{k} \geqslant 0 \}}\|_{L^{q}(Q_{k-1})} \leqslant 
\frac{2^{\frac{10k}{3q}}}{R^{\beta \frac{10}{3q}}}C_{q}U_{k-1}^{\frac{5}{3q}}$, for 
$q\geqslant 1$ provides us with the term $\frac{1}{R^{\frac{10 \beta }{3q}}}$ which decays to $0$ in a way much faster than $\frac{1}{R^{\frac{10}{3q}}}$ as $R\rightarrow \infty$, and this is the reason why we use the cutting functions $w_{k}$ instead of $v_{k}$ in carrying out the decomposition of the pressure $P$. \\

Let us first show that   Proposition \ref{index}
provides the result of Proposition \ref{bounded}. First, we  show that
the sequence $\{U_k\}_{k\geqslant 1}$ converges to 0, when $k$ goes to infinity. We can
use for instance the following easy lemma (see \cite{Vasseur}):

\begin{lem}\label{Vass}
For any given constants $B$, $\beta > 1$, there exists some constant
$C^*_{0}$ such that for any sequence $\{a_{k}\}_{k\geqslant 1}$
satisfying $0 < a_{1} \leq C^*_{0}$ and $a_{k} \leqslant B^k
a^{\beta}_{k-1}$, for any $k \geqslant 1$, we have
$lim_{k\rightarrow \infty} a_{k} = 0$ .
\end{lem}

With the assistance of lemma~\ref{Vass}, we will derive the conclusion of proposition~\ref{bounded} from proposition~\ref{index} in the following way. Let $u$ be a suitable weak solution which satisfies the hypothesis of proposition~\ref{bounded}. Then, according to the conclusion of proposition~\ref{index}, we know that if the number $p$ with $1<p<\frac{5}{4}$ is chosen to be sufficiently close to $1$, and if the number $\beta > \frac{6-3p}{10-8p}$ is chosen to be sufficiently close to $\frac{3}{2}$, it follows that the sequence $\{U_{k}\}_{k=1}^{\infty }$ will satisfies the following inequality

\begin{equation}
U_{k} \leqslant \frac{D}{R^{\Phi (p,\beta , \gamma )}}2^{\frac{10k}{3}}
\{ U_{k-1}^{\frac{5-p}{3p}} + U_{k-1}^{\frac{5}{3p}} + U_{k-1}^{\frac{5}{3}}\} ,
\end{equation}

in which $D$ stands for some positive constant which depends on the choice of the suitable weak solution  $u$ but independent of $R$, and $\Phi (p, \beta , \gamma )$ is some positive index which depends only on $p$, $\beta$, and $\gamma$. Now,
let us apply Lemma~\ref{Vass} to deduce that there is some constant $C_{0}^{*}$
, such that for any sequence $\{a_{k}\}_{k=1}^{\infty }$ satisfying 
$0 < a_{1} \leqslant C_{0}^{*}$ and $a_{k}\leqslant 2^{\frac{10k}{3}}a_{k-1}^{\frac{5-p}{3p}}$for all $k\geqslant 1$, we have $lim_{k\rightarrow \infty}a_{k} = 0$. We then choose $R > 1$ to be sufficiently large, so that we have $\frac{3D}{R^{\Phi (p, \beta ,\gamma )}} < 1$, and that 
$U_{1} \leqslant min\{1, C_{0}^{*}\}$. With this suitable choice of $R$, we see that the sequence $\{U_{k}\}_{k=0}^{\infty }$ will satisfies the conditions that
$U_{1} \leqslant C_{0}^{*}$ and $U_{k} \leqslant 2^{\frac{10k}{3}}U_{k-1}^{\frac{5-p}{3p}}$, for all $k\geqslant 1$. Hence it follows that 
$lim_{k\rightarrow \infty}U_{k} = 0$. However, because for almost every $t\in [\frac{3}{4} , 1]$, we have 

\begin{equation*}
\int_{\mathbb{R}^{3}}|u(t,x)-R|^{2}dx \leqslant 2 lim_{k\rightarrow \infty }U_{k} = 0 .
\end{equation*}

It follows at once that $|u|\leqslant R$, almost everywhere over $[\frac{3}{4}, 1]\times \mathbb{R}^{3}$. This indicates that $u$ is essentially bounded over $[\frac{3}{4}, 1]\times \mathbb{R}^{3}$. Hence, we see that the conclusion of proposition~\ref{bounded} follows provided that proposition~\ref{index} is indeed valid.\\

 For this reason, the main task of this paper is to give a detailed proof of proposition~\ref{index}, which is what we will achieve in the following sections. More precisely, after we have given some preliminaries in section 3, we will actually carry out the proof of proposition~\ref{index} in section 4. Moreover, the proof of proposition~\ref{index} as presented in section 4 will be splitted into five successive steps. In step one, we will derive the inequality of the level set energy which gives an estimate of $U_{k}$ with respect to the pressure term
$\int_{T_{k-1}}^{1}|\int_{\mathbb{R}^{3}}\frac{v_{k}}{|u|}u\nabla P dx|ds$. In step two, we will decompose the pressure $P$ into $P = P_{k1}+P_{k2} + P_{k3}$ by using the cutting functions $w_{k} = \{|u|-R^{\beta }(1-\frac{1}{2^{k}})\}_{+}$, with $\beta > \frac{3}{2}$ to be some sutiable index sufficiently close to $\frac{3}{2}$ (for more detail see equations (\ref{P1}), (\ref{P2}), and (\ref{P3})). Here, we  remark that $P_{k2}$ and $P_{k3}$ represent the effect of large velocity values $|u|\chi_{\{|u|\geqslant R^{\beta}(1-\frac{1}{2^{k}})\}}$ on the pressure, while $P_{k1}$ represents the effect of those velocity values smaller than $R^{\beta}(1-\frac{1}{2^{k}})$ on the pressure. Step three is didicated to the control of the two pressure terms involving big velocity values. Thanks to the introduction of the cutting functions $w_{k} = \{|u|-R^{\beta }(1-\frac{1}{2^{k}})\}_{+}$ in the decomposition of the pressure, the control on these two terms  can then be performed successfully. In step four and step five, we will control the pressure term 
$\int_{T_{k-1}}^{1}|\int_{\mathbb{R}^{3}} \nabla (\frac{v_{k}}{|u|})uP_{k1}dx|ds$ which depends on those velocity values smaller than $R^{\beta }(1-\frac{1}{2^{k}})$. In step four, we will show that such a pressure term depending on those velocity values smaller than $R^{\beta }(1-\frac{1}{2^{k}})$ can be controlled by a weighted $|F|log^{+}|F|$ norm of $div(\frac{u}{|u|})$. We will finally show in step five that, in some specific way, we can eventually control the pressure term 
$\int_{T_{k-1}}^{1}|\int_{\mathbb{R}^{3}}\nabla (\frac{v_{k}}{|u|})u P_{k1}dx |ds$
successfully by employing the hypothesis 
$\frac{|u\cdot \nabla F|}{|u|^{\gamma}} \leqslant A |F|$ of proposition~\ref{index}.

\section{Preliminaries for the proof of proposition 2.1}

\begin{lem}
There exists some constant $C > 0$, such that for any $k\geqslant 1$, and any
$f\in L^{\infty} (T_{k} , 1 ; L^2(\mathbb{R}^3))$ with $\nabla f \in L^2(Q_{k})$, we have
$\|f\|_{L^{\frac{10}{3}}(Q_{k})}\leqslant C \|f\|_{L^{\infty}(T_{k} , 1 ; L^2(\mathbb{R}^3))}^{\frac{2}{5}} \|\nabla
f\|_{L^2(Q_{k})}^{\frac{3}{5}}$.
\end{lem}

\begin{proof}
By Sobolev-embedding Theorem, there is a constant $C$, depending only on the dimension of $\mathbb{R}^3$, such that
\begin{equation*}
(\int_{\mathbb{R}^3} |f(t,x)|^6 dx)^{\frac{1}{6}} \leqslant C (\int_{\mathbb{R}^3}|\nabla f(t,x)|^2
dx)^{\frac{1}{2}} .
\end{equation*}

for any $t\in [T_{k} , 1]$, where $k\geqslant 1$, and $f$ is some function which verifies $f\in L^{\infty}(T_{k},1 ;
L^2(\mathbb{R}^3))$, and $\nabla f \in L^2(Q_{k})$. By taking power $2$ on both sides of the above inequality and
then taking integration along the variable $t\in [T_{k}, 1]$, we yield

\begin{equation*}
\int_{T_{k}}^{1} (\int_{\mathbb{R}^3} |f|^6 dx)^{\frac{1}{3}} dt \leqslant C^2
\int_{T_{k}}^{1}\int_{\mathbb{R}^3} |\nabla f|^2 dx\,dt .
\end{equation*}

On the other hand, by Holder's inequality, we have
\begin{equation*}
\begin{split}
\|f\|_{L^{\frac{10}{3}}(Q_{k})}^{\frac{10}{3}} =
\int_{T_{k}}^{1} \int_{\mathbb{R}^3} |f|^2 |f|^{\frac{4}{3}} dx\,dt\\
& \leqslant
\int_{T_{k}}^{1} (\int_{\mathbb{R}^3} |f|^6 dx)^{\frac{1}{3}}(\int_{\mathbb{R}^3}|f|^2 dx)^{\frac{2}{3}} dt\\
& \leqslant \|f\|_{L^{\infty}(T_{k} , 1 ; L^2(\mathbb{R}^3))}^{\frac{4}{3}} \|f\|_{L^2(T_{k}, 1 ;
L^6(\mathbb{R}^3))}^2.
\end{split}
\end{equation*}

By taking the advantage that $\|f\|_{L^2(T_{k} ,1 ; L^6(\mathbb{R}^3) )} \leqslant C \|\nabla f\|_{L^2(Q_{k})}$, we yield
\begin{equation*}
\|f\|_{L^{\frac{10}{3}}(Q_{k})}^{\frac{10}{3}}
\leqslant C^2 \|f\|_{L^{\infty}(T_{k} , 1 ; L^2(\mathbb{R}^3))}^{\frac{4}{3}} \|\nabla f\|_{L^2(Q_{k})}^2 .
\end{equation*}

Hence, we have

\begin{equation*}
\|f\|_{L^{\frac{10}{3}}(Q_{k})} \leqslant C \|f\|_{L^{\infty}(T_{k} , 1 ; L^2(\mathbb{R}^3))}^{\frac{2}{5}}
\|\nabla f\|_{L^2(Q_{k})}^{\frac{3}{5}} .
\end{equation*}
so, we are done
\end{proof}

\begin{lem}\label{cheb}
For any $1 < q < \infty$, we have $\|\chi_{\{ v_{k} \geqslant 0 \} }\|_{L^q(Q_{k-1})} \leqslant 
\frac{2^{\frac{10k}{3q}}}{R^{\frac{10}{3q}}}
C^{\frac{1}{q}}U_{k-1}^{\frac{5}{3q}}$ .
\end{lem}

\begin{proof}
First, we have to notice that $\{v_{k} \geqslant 0\}$ is a subset of $\{v_{k-1} \geqslant \frac{R}{2^k}\}$, hence we
have
\begin{equation*}
\int_{Q_{k-1}} \chi_{\{ v_{k} > 0\}} \leqslant \int_{Q_{k-1}} \chi_{\{v_{k-1} >  \frac{R}{2^k}\}} \leqslant
\frac{2^{\frac{10k}{3}}}{R^{\frac{10}{3}}} \int_{Q_{k-1}} |v_{k-1}|^{\frac{10}{3}} .
\end{equation*}

By our previous Lemma, we have

\begin{equation*}
\begin{split}
\|v_{k-1}\|_{L^{\frac{10}{3}}(Q_{k-1})}^{\frac{10}{3}} \leqslant \\
&C^2 \|v_{k-1}\|_{L^{\infty}(T_{k-1},1;L^2(\mathbb{R}^3))}^{\frac{4}{3}}
\|\nabla v_{k-1} \|_{L^2(Q_{k-1})}^{2} \\
& \leqslant C^2 (U_{k-1}^{\frac{1}{2}})^{\frac{4}{3}} \|d_{k-1}\|_{L^2(Q_{k-1})}^{2} \\
& \leqslant C^2 U_{k-1}^{\frac{2}{3}} U_{k-1} \\
& = C^2 U_{k-1}^{\frac{5}{3}} .
\end{split}
\end{equation*}

So, it follows that $\int_{Q_{k-1}} \chi_{\{v_{k} > 0\}} \leqslant 
\frac{ 2^{\frac{10k}{3}}}{R^{\frac{10}{3}}} C^2
U_{k-1}^{\frac{5}{3}}$, and hence we have
$\|\chi_{\{  v_{k} \geqslant 0 \}}\|_{L^q(Q_{k-1})} \leqslant 
\frac{ 2^{\frac{10k}{3q}}}{R^{\frac{10}{3q}}}  C^{\frac{1}{q}} U_{k-1}^{\frac{5}{3q}}$,
where C is some  universal constant. So, we are done.
\end{proof}

Just as we have said before, we will need to decompose the pressure by employing the sequence of cutting functions $w_{k} = \{|u| - R^{\beta }( 1-\frac{1}{2^{k}})\}_{+}$, for $k \geqslant 1$. We also said that we prefer to do this because the cutting functions $w_{k}$ satisfies the following inequality which can be justified in the same way as lemma~\ref{cheb}.\\

\begin{lem}\label{cheb2}
For every $q\geqslant 1$, we have $\|\chi_{\{w_{k}\geqslant 0 \}}\|_{L^{q}(Q_{k-1})}
\leqslant \frac{1}{R^{\frac{10\beta }{3q}}}2^{\frac{10k}{3q}}C_{q}U_{k-1}^{\frac{5}{3q}}$, for all $k\geqslant 1$, in which $C_{q}$ is some constant depending only on $q$.
\end{lem}

Indeed, in dealing with the pressure terms, we will invoke the lemma~\ref{cheb2} without explicit mention.\\

In the proof of Lemma~\ref{cheb}, we have used the fact that $|\nabla v_{k}| \leqslant d_{k}$, whose
justification will be given immediately in the following paragraph.\\
Before we leave this section, we also want to list out some inequalities which will often be used in the proof
of proposition~\ref{index} as follow:

\begin{itemize}
\item $|(1-\frac{v_{k}}{|u|})u| \leqslant R( 1-\frac{1}{2^k})$.
\item $\frac{v_{k}}{|u|}|\nabla u| \leqslant d_{k}$.
\item $\chi_{\{v_{k} \geqslant 0\}}|\nabla |u|| \leqslant d_{k} $.
\item $|\nabla v_{k}|\leqslant d_{k}$.
\item $|\nabla (\frac{v_{k}}{|u|}u)| \leqslant 3d_{k}$.
\end{itemize}

Now, we first want to justify the validity of
$|(1-\frac{v_{k}}{|u|})u|\leqslant R( 1-\frac{1}{2^k})$. In the case in
which the point $(t,x)$ satisfies $|u(t,x)|< R( 1-\frac{1}{2^k})$, we
have $v_{k}(t,x) = 0$, and hence it follows that
\begin{equation*}
|\{1-\frac{v_{k}(t,x)}{|u(t,x)|}\}u(t,x)| = |u(t,x)| < R( 1-\frac{1}{2^k}) .
\end{equation*}
 In the case in which $(t,x)$ satisfies $|u(t,x)|\geqslant R( 1-\frac{1}{2^k})$, we have
 $v_{k}(t,x) = |u(t,x)| - R(1-\frac{1}{2^k})$, and hence it follows that

\begin{equation*}
|\{ 1 - \frac{v_{k}}{|u|}\} u(t,x)| = |1-\frac{|u|-R(1-\frac{1}{2^k})}{|u|}||u|=R( 1-\frac{1}{2^k}) .
\end{equation*}

So, no matter in which case, we always have the conclusion that $|(1-\frac{v_{k}}{|u|})u|\leqslant
 R( 1-\frac{1}{2^k})$.\\
Next, according to the definition of $d_{k}^2$, we can carry out the following estimation
\begin{equation*}
d_{k}^2 \geqslant \frac{v_{k}}{|u|} |\nabla u|^2 \geqslant \{\frac{v_{k}}{|u|} |\nabla u|\}^2 .
\end{equation*}
Hence, by taking square root, it follows at once that $d_{k} \geqslant \frac{v_{k}}{|u|} |\nabla u|$.\\
We now turn our attention to the inequality
$\chi_{\{ |u|\geqslant R(1-\frac{1}{2^k})\}   }|\nabla |u|| \leqslant d_{k}$. To justify
it, we recall that $|\nabla u|\geqslant |\nabla |u||$. Hence, it follows from the definition of $d_{k}^2$ that

\begin{equation*}
d_{k}^2 \geqslant \frac{R( 1-\frac{1}{2^k})}{|u|}\chi_{\{|u|\geqslant R( 1-\frac{1}{2^k})  \}} |\nabla |u||^2
+ \{1-\frac{R( 1-\frac{1}{2^k}) }{|u|}\}\chi_{\{|u|\geqslant  R( 1-\frac{1}{2^k})  \}}|\nabla |u||^2 .
\end{equation*}

So, by simplifying the right-hand side of the above inequality, we can deduce that
$d_{k}^2 \geqslant \chi_{\{|u|\geqslant R(1-\frac{1}{2^k})   \}} |\nabla|u||^2$. Hence, we have
$d_{k}\geqslant \chi_{\{|u|\geqslant R( 1-\frac{1}{2^k})  \}} |\nabla |u||$. In addition, since it is obvious to see
that $\nabla v_{k} = \chi_{\{|u|\geqslant R( 1-\frac{1}{2^k}) \}} \nabla |u|$, we also have the result that
$|\nabla v_{k}|\leqslant d_{k}$.\\
Finally, we want to justify the inequality that $|\nabla (\frac{v_{k}}{|u|}u)|\leqslant 3d_{k}$. So, we
notice that, by applying the product rule, we have

\begin{equation*}
\nabla (\frac{v_{k}}{|u|} u) = \nabla (v_{k})\frac{u}{|u|} + \frac{v_{k}}{|u|}\nabla u
- \frac{v_{k}}{|u|^2}u\nabla |u| .
\end{equation*}

However, since $\frac{v_{k}}{|u|}|\nabla u| \leqslant d_{k}$, and
$|\frac{v_{k}}{|u|^2}u\nabla |u|| \leqslant \chi_{\{|u|\geqslant R( 1-\frac{1}{2^k}) \}} |\nabla |u|| \leqslant
d_{k}$, it follows at once from the above expression that $|\nabla (\frac{v_{k}}{|u|}u)|\leqslant 3d_{k}$.

\section{proof of proposition 2.1}

\vskip0.2cm

\noindent{\bf Step one}

To begin the argument, we recall that, by multiplying the equation
$\partial_{t} u - \triangle u + div( u \otimes u) + \nabla P = 0$ by the term $\frac{v_{k}}{|u|}u$, we
yield the following inequality formally, which is indeed valid in the sense of distribution

\begin{equation*}
\partial_{t} (\frac{v_{k}^2}{2}) + d_{k}^{2} - \triangle (\frac{v_{k}^2}{2}) + div(\frac{v_{k}^2}{2} u) +
\frac{v_{k}}{|u|} u \nabla P \leqslant 0 .
\end{equation*}

Next, let us consider the variables $\sigma$ , $t$ verifying $T_{k-1} \leqslant \sigma \leqslant T_{k} \leqslant t
\leqslant 1$. Then, we have

\begin{itemize}
\item $\int_{\sigma}^{t} \int_{\mathbb{R}^3} \partial_{t} (\frac{v_{k}^2}{2}) dx\,ds =
\int_{\mathbb{R}^3} \frac{v_{k}^2(t,x)}{2} dx - \int_{\mathbb{R}^3} \frac{v_{k}^2(\sigma ,x)}{2} dx$.
\item $\int_{\sigma}^{t}\int_{\mathbb{R}^3} \triangle (\frac{v_{k}^2}{2}) dx\,ds = 0$.
\item $\int_{\sigma}^{t} \int_{\mathbb{R}^3} div (\frac{v_{k}^2}{2} u) dx\,ds = 0$.
\end{itemize}

So, it is straightforward to see that

\begin{equation*}
\int_{\mathbb{R}^3} \frac{v_{k}^2(t,x)}{2} dx + \int_{\sigma}^{t}\int_{\mathbb{R}^3}d_{k}^2dx\,ds
\leqslant \int_{\mathbb{R}^3} \frac{v_{k}^2(\sigma ,x)}{2}dx +
\int_{\sigma}^{t}    \vert \int_{\mathbb{R}^3}\frac{v_{k}}{|u|} u \nabla P dx \vert       ds ,
\end{equation*}

for any $\sigma$, $t$ satisfying $T_{k-1}\leqslant \sigma \leqslant T_{k} \leqslant t \leqslant 1$. By taking the
average over the variable $\sigma$, we yield

\begin{equation*}
\int_{\mathbb{R}^3} \frac{v_{k}^2(t,x)}{2} dx + \int_{T_{k}}^{t}\int_{\mathbb{R}^3}d_{k}^2 dx\,ds
\leqslant \frac{4^{k+1}}{6}   \int_{T_{k-1}}^{T_{k}}\int_{\mathbb{R}^3}v_{k}^2(s,x)dx\,ds
+ \int_{T_{k-1}}^{t}|\int_{\mathbb{R}^3}\frac{v_{k}}{|u|} u \nabla Pdx|ds .
\end{equation*}

By taking  the sup over $t\in [T_{k} , 1]$. the above inequality will give the following

\begin{equation*}
U_{k} \leqslant \frac{4^{k+1}}{6}   \int_{Q_{k-1}}v_{k}^2 + \int_{T_{k-1}}^{1}|\int_{\mathbb{R}^3}\frac{v_{k}}{|u|}u\nabla
Pdx|ds .
\end{equation*}

But, from Lemma~\ref{cheb} and Holder's inequality, we have

\begin{equation*}
\begin{split}
\int_{Q_{k-1}}v_{k}^2 = \int_{Q_{k-1}}v_{k}^2 \chi_{\{v_{k} \geqslant 0\}} \\
& \leqslant (\int_{Q_{k-1}}v_{k}^{\frac{10}{3}})^{\frac{3}{5}} \|\chi_{\{v_{k} \geqslant 0 \}}\|_{L^{\frac{5}{2}}(Q_{k-1})}\\
& \leqslant \|v_{k}\|_{L^{\frac{10}{3}}(Q_{k-1})}^{2} \frac{2^{\frac{4k}{3}}}{R^{\frac{4}{3}}}      C^{\frac{2}{5}}U_{k-1}^{\frac{2}{3}}\\
& \leqslant \|v_{k-1}\|_{L^{\frac{10}{3}}(Q_{k-1})}^{2} \frac{ 2^{\frac{4k}{3}}}{R^{\frac{4}{3}}}           
C^{\frac{2}{5}}U_{k-1}^{\frac{2}{3}}\\
& \leqslant CU_{k-1}^{\frac{5}{3}} \frac{ 2^{\frac{4k}{3}}}{R^{\frac{4}{3}}} .
\end{split}
\end{equation*}

As a result, we have the following conclusion

\begin{equation}
U_{k}\leqslant \frac{ 2^{\frac{10k}{3}}}{R^{\frac{4}{3}}}     C U_{k-1}^{\frac{5}{3}} +
\int_{T_{k-1}}^{1}|\int_{\mathbb{R}^3}\frac{v_{k}}{|u|}u \nabla p dx|ds .
\end{equation}

\noindent{\bf Step two}

Now, in order to estimate the term $\int_{T_{k-1}}^{1}|\int_{\mathbb{R}^3}\frac{v_{k}}{|u|}u\nabla Pdx|ds$, we would
like to carry out the following computation

\begin{equation*}
\begin{split}
-\triangle P = \sum \partial_{i} \partial_{j} (u_{i}u_{j})\\
& = \sum \partial_{i} \partial_{j} \{ (1-\frac{w_{k}}{|u|})u_{i} (1-\frac{w_{k}}{|u|})u_{j} \}
+  2\sum \partial_{i} \partial_{j} \{ (1-\frac{w_{k}}{|u|}) u_{i} \frac{w_{k}}{|u|} u_{j}  \}\\
& + \sum \partial_{i} \partial_{j} \{ \frac{w_{k}}{|u|} u_{i} \frac{w_{k}}{|u|} u_{j} \} ,
\end{split}
\end{equation*}

in which $w_{k}$ is given by $w_{k} = \{|u| - R^{\beta} (1- \frac{1}{2^{k}} )\}_{+}$, and $\beta$ is simply the arbritary index involved in proposition~\ref{index}. This motivates us to decompose $P$ as $P = P_{k1} + P_{k2} + P_{k3} $, in which

\begin{equation}\label{P1}
-\triangle P_{k1} = \sum \partial_{i} \partial_{j} \{ (1-\frac{w_{k}}{|u|})u_{i}
 (1-\frac{w_{k}}{|u|})u_{j}  \} ,
\end{equation}

\begin{equation}\label{P2}
- \triangle P_{k2} = \sum \partial_{i} \partial_{j} \{ 2(1-\frac{w_{k}}{|u|})u_{i} \frac{w_{k}}{|u|}
 u_{j} \}
\end{equation}

\begin{equation}\label{P3}
-\triangle P_{k3} = \sum \partial_{i} \partial_{j}\{\frac{w_{k}}{|u|}u_{i} \frac{w_{k}}{|u|}u_{j} \} .
\end{equation}

Here, we have to remind ourself that the cutting functions which are used in the decomposition of the pressure are indeed $w_{k} = \{|u| - R^{\beta} (1- \frac{1}{2^{k}}) \}_{+}$, for all $k \geqslant 0$ , in which $\beta$ is some suitable index strictly greater than $\frac{3}{2}$. With respect to the cutting functions $w_{k}$, we need to define the respective $D_{k}$ as follow:

\begin{equation*}
D_{k}^{2} = \frac{R^{\beta} (1-\frac{1}{2^{k}})}{|u|} \chi_{\{ w_{k} \geqslant 0 \}} |\nabla |u||^{2}
+ \frac{w_{k}}{|u|} |\nabla u|^{2} .
\end{equation*}

Then, just like what happens to the cutting functions $v_{k}$, we have the following assertions about the cutting functions $w_{k}$, which are easily verified.

\begin{itemize}
\item $|\nabla w_{k}| \leqslant D_{k}$, for all $k \geqslant 0$.
\item $|\nabla (\frac{w_{k}}{|u|}u_{i})| \leqslant 3 D_{k}$, for all $k \geqslant 0$, and $1 \leqslant i \leqslant 3$.
\item $ |\nabla (\frac{w_{k}}{|u|}) u_{i} | \leqslant 2 D_{k} $, for any $k \geqslant 0$, and
$1 \leqslant i \leqslant 3 $.
\end{itemize}

Besides these, we also need the following lemma which links $D_{k}$ to $d_{k}$.

\begin{lem}\label{D}
There is some sufficiently large $R_{0} > 1$, such that whenever $R > R_{0}$ and $k \geqslant 1$, we have $D_{k} \leqslant 5^{\frac{1}{2}} d_{k}$.
\end{lem}

\begin{proof}
since $\frac{R^{\beta } - R}{R^{\beta } }$ trends to the limiting value $1$, as $R$ trends to $\infty$. So, there is some sufficiently large $R_{0} >1$ for which $(R^{\beta } - R) > \frac{R^{\beta }}{2} $, for all $R > R_{0}$. Now, notice that $\{w_{k} \geqslant 0  \} $ is a subset of $\{ v_{k} \geqslant (R^{\beta } - R    ) (1-\frac{1}{2^{k}}) \}$, for all $k \geqslant 0$. Hence, it follows that $\{ w_{k} \geqslant 0 \}$ is a subset of $\{v_{k} > \frac{R^{\beta }}{4} \}$, for all $k\geqslant 1$ and $R > R_{o}$. As a result, we can carry out the following computation

\begin{equation*}
\begin{split}
D_{k}^2 = \frac{R^{\beta}(1-\frac{1}{2^k} )}{|u|} \chi_{\{w_{k} \geqslant 0 \}} |\nabla |u||^{2} + \frac{w_{k}}{|u|} |\nabla u|^{2} \\
&\leqslant \frac{R^{\beta }}{|u|} \chi_{\{ w_{k} \geqslant 0 \}} |\nabla |u||^2 + \frac{v_{k}}{|u|} |\nabla u|^{2} \\
&\leqslant \frac{4v_{k}}{|u|} \chi_{\{w_{k} \geqslant 0 \}} |\nabla |u||^{2} + \frac{v_{k}}{|u|} |\nabla u|^{2}\\
&\leqslant \frac{5v_{k}}{|u|} |\nabla u|^2 \leqslant 5 d_{k}^{2} ,
\end{split}
\end{equation*}

for any $k \geqslant 1 $, and $R > R_{0}$. Hence, we have $D_{k} \leqslant 5^{\frac{1}{2}} d_{k}$, for all $k \geqslant 1$, and all $R > R_{0}$. So, we are done.
\end{proof}

Now, let us recall that we have already used the cutting functions $w_{k}$ to obtain the decomposition $P = P_{k1} + P_{k2} + P_{k3}$, in which $P_{k1}$, $P_{k2}$, and$P_{k3}$ are described in equations (\ref{P1}), (\ref{P2}), and (\ref{P3}) respectively.

Hence, we have

\begin{equation*}
\begin{split}
\int_{T_{k-1}}^{1} |\int_{\mathbb{R}^{3}} \frac{v_{k}}{|u|} u \nabla P dx  |dt\\
&\leqslant \int_{T_{k-1}}^{1} |\int_{\mathbb{R}^{3}} \nabla (\frac{v_{k}}{|u|}) u P_{k1} dx |dt +\int_{Q_{k-1}} (1-\frac{v_{k}}{|u|})|u| |\nabla P_{k2}|\\
& + \int_{Q_{k-1}} (1-\frac{v_{k}}{|u|})|u| |\nabla P_{k3}| .
\end{split}
\end{equation*}

\noindent{\bf Step 3}

We are now ready to deal with the term  $\int_{Q_{k-1}}(1-\frac{v_{k}}{|u|})|u||\nabla P_{k2}|$. For this purpose, let $p$ be such that $1< p < \frac{5}{4} $, and let $q = \frac{p}{p-1}$, so that $2<q<\infty$. By applying Holder's inequality, we find that

\begin{equation*}
\begin{split}
\|(1-\frac{v_{k}}{|u|})u \|_{L^{q}(\mathbb{R}^{3})}\\
& \leqslant
\|(1-\frac{v_{k}}{|u|})u\|_{L^{2}(\mathbb{R}^{3})}^{\frac{2}{q}}
\|(1-\frac{v_{k}}{|u|})u\|_{L^{\infty}(\mathbb{R}^{3})}^{1-\frac{2}{q}}\\
&\leqslant R^{1-\frac{2}{q}} \|(1-\frac{v_{k}}{|u|})u\|_{L^{2}(\mathbb{R}^{3})}^{\frac{2}{q}}\\
&\leqslant  R^{\frac{2}{p}-1} 
\|u\|_{L^{\infty}(0,1;L^{2}(\mathbb{R}^{3}))}^{2(1-\frac{1}{p})}
\end{split}
\end{equation*}

Hence, it follows from Holder's inequality that

\begin{equation*}
\int_{\mathbb{R}^{3}} (1-\frac{v_{k}}{|u|})|u||\nabla P_{k2}| dx 
\leqslant R^{\frac{2}{p}-1} 
\|u\|_{L^{\infty}(0,1;L^{2}(\mathbb{R}^{3}))}^{2(1-\frac{1}{p})} 
\{\int_{\mathbb{R}^{3}} |\nabla P_{k2}|^{p} dx \}^{\frac{1}{p}} .
\end{equation*}

Hence, we have

\begin{equation}\label{4}
\int_{Q_{k-1}} (1-\frac{v_{k}}{|u|})|u||\nabla P_{k2}| 
\leqslant R^{\frac{2}{p}-1}
\|u\|_{L^{\infty}(0,1;L^{2}(\mathbb{R}^{3}))}^{2(1-\frac{1}{p})}
\|\nabla P_{k2}\|_{L^{p}(Q_{k-1})} .
\end{equation}

But, we recognize that 

\begin{equation*}
\nabla P_{k2} = \sum R_{i}R_{j}\{2(1-\frac{w_{k}}{|u|})u_{i} \nabla [\frac{w_{k}}{|u|}u_{j}] 
+ 2(1-\frac{w_{k}}{|u|})u_{j}[\frac{w_{k}}{|u|}\nabla u_{i}]
- 2 \nabla [\frac{w_{k}}{|u|}]u_{i} \frac{w_{k}}{|u|}u_{j}\}.
\end{equation*}

Moreover, it is straightforward to see that for any $1 \leqslant i,j \leqslant 3$, we have

\begin{itemize}
\item $ |2(1-\frac{w_{k}}{|u|})u_{i} \nabla [\frac{w_{k}}{|u|}u_{j}] +  2(1-\frac{w_{k}}{|u|})u_{j}[\frac{w_{k}}{|u|}\nabla u_{i}]| \leqslant 8 R^{\beta } D_{k}$.
\item $|2 \nabla [\frac{w_{k}}{|u|}]u_{i} \frac{w_{k}}{|u|}u_{j}| \leqslant 8 w_{k} D_{k}$.
\end{itemize}

So, we can decompose $\nabla P_{k2}$ as $\nabla P_{k2} = G_{k21} + G_{k22}$, where $G_{k21}$ and $G_{k22}$ are given by

\begin{itemize}
\item $G_{k21} = \sum R_{i} R_{j} \{2(1-\frac{w_{k}}{|u|})u_{i} \nabla [\frac{w_{k}}{|u|}u_{j}] +   2(1-\frac{w_{k}}{|u|})u_{j}[\frac{w_{k}}{|u|}\nabla u_{i}]\}$.
\item $G_{k22} = -\sum R_{i} R_{j} \{2 \nabla [\frac{w_{k}}{|u|}]u_{i} \frac{w_{k}}{|u|}u_{j}\}$.
\end{itemize}

In order to use inequality (\ref{4}), we need to estimate $\|G_{k21}\|_{L^{p}(Q_{k-1})}$ and $\|G_{k22}\|_{L^{p}(Q_{k-1})}$ respectively, for $p$ with $1 < p < \frac{5}{4} $. Indeed, by applying the Zygmund-Calderon Theorem, we can deduce that

\begin{itemize}
\item $\|G_{k21}\|_{L^{p}(Q_{k-1})} \leqslant C_{p} R^{\beta } \|D_{k}\|_{L^{p}(Q_{k-1})}$,
\item $ \|G_{k22}\|_{L^{p}(Q_{k-1})} \leqslant C_{p} \|w_{k}D_{k}\|_{L^{p}(Q_{k-1})}$,
\end{itemize}

 where $C_{p}$ is some constant depending only on $p$. But it turns out that

\begin{equation*}
\begin{split}
\|D_{k}\|_{L^{p}(Q_{k-1})}^{p} = \int_{Q_{k-1}} D_{k}^{p} \chi_{\{w_{k} \geqslant 0 \}}\\
& \leqslant \{\int_{Q_{k-1}}D_{k}^{2}\}^{\frac{p}{2}} 
\|\chi_{\{w_{k} \geqslant 0 \}}\|_{L^{\frac{2}{2-p}}(Q_{k-1})}\\
&\leqslant  \frac{5^{\frac{p}{2}}}{R^{\frac{5}{3}\beta (2-p)}}\|d_{k}\|_{L^{2}(Q_{k-1})}^{p} 
2^{\frac{5k}{3}(2-p)} C_{p} U_{k-1}^{\frac{5}{6}(2-p)}\\
& \leqslant \frac{1}{R^{\frac{5}{3}\beta (2-p)}} C_{p} U_{k-1}^{\frac{5-p}{3}}
2^{\frac{5(2-p)k}{3}} .
\end{split}
\end{equation*}

That is , we have

\begin{equation*}
\|D_{k}\|_{L^{p}(Q_{k-1})} \leqslant \frac{1}{R^{\frac{5}{3p}\beta (2-p) }}
C_{p} U_{k-1}^{\frac{5-p}{3p}} 2^{\frac{5(2-p)k}{3p}} .
\end{equation*}

Hence, it follows that

\begin{equation}\label{5}
\|G_{k21}\|_{L^{p}(Q_{k-1})} \leqslant \frac{1}{R^{\beta (\frac{10-8p}{3p})}}
C_{p} U_{k-1}^{\frac{5-p}{3p}} 2^{\frac{5(2-p)k}{3p}} .
\end{equation}

On the other hand, we have

\begin{equation*}
\begin{split}
\|w_{k}D_{k}\|_{L^{p}(Q_{k-1})}^{p} = \int_{Q_{k-1}} w_{k}^{p} D_{k}^{p}\\
& \leqslant \{ \int_{Q_{k-1}}w_{k}^{\frac{2p}{2-p}}\}^{\frac{2-p}{2}}
\{\int_{Q_{k-1}}D_{k}^2  \}^{\frac{p}{2}}\\
& \leqslant C_{p} \{\int_{Q_{k-1}} w_{k}^{\frac{2p}{2-p}}\}^{\frac{2-p}{2}} 
U_{k-1}^{\frac{p}{2}} .
\end{split}
\end{equation*}

Now, let us recall that $1 < p < \frac{5}{4}$, and put $r = \frac{2p}{2-p}$. we then recognize that $2 < r = \frac{2p}{2-p} < \frac{10}{3}$, if $1 < p < \frac{5}{4}$. So, we can have the following estimation

\begin{equation*}
\begin{split}
\int_{Q_{k-1}}w_{k}^{\frac{2p}{2-p}}
 = \int_{Q_{k-1}} w_{k}^{r} \chi_{\{ w_{k} \geqslant 0  \}}\\
&\leqslant \int_{Q_{k-1}} w_{k}^{r} \chi_{\{ w_{k-1} \geqslant \frac{R^{\beta}}{2^{k}}\}}\\
&\leqslant \frac{1}{R^{\beta (\frac{10}{3} -r)}} 2^{k( \frac{10}{3} - r )}
\int_{Q_{k-1}} w_{k}^{\frac{10}{3}}\\
&\leqslant \frac{1}{R^{\beta \frac{20-16p}{3(2-p)} }} 2^{\frac{k(20-16p)}{3(2-p)}} 
U_{k-1}^{\frac{5}{3}} .
\end{split}
\end{equation*}

Hence, it follows that

\begin{equation}\label{6}
\|G_{k22}\|_{L^{p}(Q_{k-1})} \leqslant  \|w_{k}D_{k}\|_{L^{p}(Q_{k-1})} \leqslant C_{p} 
\frac{2^{k \frac{10-8p}{3p}}}{R^{\beta \frac{10-8p}{3p}}} U_{k-1}^{\frac{5-p}{3p}} .
\end{equation}

By combining inequalities (\ref{4}), (\ref{5}), (\ref{6}), we deduce that

\begin{equation}\label{middle}
\begin{split}
\int_{Q_{k-1}} (1-\frac{v_{k}}{|u|})|u||\nabla P_{k2}|\\
&\leqslant \frac{1}{R^{\beta \frac{10-8p}{3p} - \frac{2-p}{p}}}
C_{p} \|u\|_{L^{\infty}(0,1;L^{2}(\mathbb{R}^{3}))}^{2(1-\frac{1}{p})} U_{k-1}^{\frac{5-p}{3p}}
2^{\frac{10-5p}{3p} k}  .
\end{split}
\end{equation}

Notice that $\beta (\frac{10-8p}{3p}) - (\frac{2-p}{p}) > 0$ if and only if
 $\beta > \frac{6-3p}{10-8p}$. Moreover, we know that the term $\frac{6-3p}{10-8p}$ is always positive, for $1 < p < \frac{5}{4}$. In addition, we know that as $p$ trends to $1$,$ \frac{6-3p}{10-8p}$ trends to $\frac{3}{2}$. This means that even though $\beta$ cannot be exactly $\frac{3}{2}$, $\beta > \frac{3}{2}$ can be adjusted to be as close to $\frac{3}{2}$ as we desire.\\

As for the term $\int_{Q_{k-1}} (1-\frac{v_{k}}{|u|})|u||\nabla P_{k3}|$. We first notice that

\begin{equation*}
P_{k3} = \sum R_{i}R_{j}\{\frac{w_{k}}{|u|}u_{i}\frac{w_{k}}{|u|}u_{j}\} . 
\end{equation*}

So, we know that 

\begin{equation*}
\nabla P_{k3} = \sum R_{i}R_{j} \{\nabla [\frac{w_{k}}{|u|}u_{i}]\frac{w_{k}}{|u|}u_{j}
+ \frac{w_{k}}{|u|}u_{i} \nabla [\frac{w_{k}}{|u|} u_{j} ] \} ,
\end{equation*}

with 

\begin{equation*}
| \nabla [\frac{w_{k}}{|u|}u_{i}]\frac{w_{k}}{|u|}u_{j}
+ \frac{w_{k}}{|u|}u_{i} \nabla [\frac{w_{k}}{|u|} u_{j} ] |
\leqslant 6 w_{k}D_{k} .
\end{equation*}

 So, it follows again from the Risez's theorem in the theory of singular operator that $\|\nabla P_{k3}\|_{L^{p}(\mathbb{R}^{3})} \leqslant C_{p} 
\|w_{k}D_{k}\|_{L^{p}(\mathbb{R}^{3})}$, in which $C_{p}$ is some constant depending only on $p$. So, we see that we can repeat the same type of estimation, just as what we have done to the term $\int_{Q_{k-1}}(1-\frac{v_{k}}{|u|})|u||\nabla P_{k2}|$, to conclude that

\begin{equation}\label{easy}
\begin{split}
\int_{Q_{k-1}} (1-\frac{v_{k}}{|u|})|u||\nabla P_{k3}|\\
&\leqslant R^{\frac{2}{p}-1}\|u\|_{L^{\infty}(0,1;L^{2}(\mathbb{R}^{3}))}^{2(1-\frac{1}{p})}
\|\nabla P_{k3}\|_{L^{p}(Q_{k-1})}\\
&\leqslant \frac{C_{p}\|u\|_{L^{\infty}(0,1; L^{2}(\mathbb{R}^{3}))}^{2(1-\frac{1}{p})}}{
R^{\beta \frac{10-8p}{3p} - \frac{2-p}{p}}} U_{k-1}^{\frac{5-p}{3p}} 2^{\frac{(10-5p)k}{3p}}.
\end{split}
\end{equation}

\noindent{\bf Step four}

Now, let us turn our attention to the term 
$\int_{T_{k-1}}^{1}|\int_{\mathbb{R}^{3}} \nabla (\frac{v_{k}}{|u|})uP_{k1}dx|ds$. Before we deal with the term written as above, let us recall that the weak solution $u$ that we are dealing with now is the one verifying the following condition

\begin{equation*}
\frac{|u\cdot \nabla F|}{|u|^{\gamma}} \leqslant A |F| ,
\end{equation*}

where $F = - \frac{u \cdot \nabla |u|}{|u|^2}$, and $\gamma$ is some index with $0 < \gamma < \frac{1}{3}$. We need to introduce the following classical Theorem of harmonic analysis which is due to John and Nirenberg~\cite{John}.

\begin{thm}
Let $B$ be a ball with finite radius sitting in $\mathbb{R}^{3}$. Then, there exists some constants $\alpha$, and $K$, with $0 < \alpha < \infty$, and $0 < K < \infty$, depending only on the ball $B$ and $n$, such that for any given $f \in BMO(\mathbb{R}^{n})$, we have $\int_{B} exp(\alpha \frac{|f-f_{B}|}{\|f\|_{BMO} }) \leqslant K$, where the symbol $f_{B}$ stands for the mean value of $f$ over $B$.
\end{thm}

We now need to establish the following lemma by using the theorem quoted as above.

\begin{lem}\label{BMO}
Let $B$ be a ball with finite radius sitting in $\mathbb{R}^{3}$. There exists some finite positive constants $\alpha$ and $K$,depending only on $B$, such that for every $\mu \geqslant 0$, every $f\in BMO(\mathbb{R}^3)$ with $\int_{B}fdx =0$, and $p$ with $1 < p < \frac{5}{4}$, we have
$\int_{B} \mu |f| \leqslant \frac{2p}{\alpha (p-1)} 
\{ 1+ K^{1-\frac{1}{p}}\}  \|f\|_{BMO}
\{ (\int_{B}\mu )^{\frac{1}{p}}  +  \int_{B} \mu log^{+}\mu  \}$.
\end{lem}

\begin{proof}
For any given $\mu \geqslant 0$, and any $f \in BMO(\mathbb{R}^{3})$ with $\int_{B}f dx = 0$, we do the following splitting

\begin{equation*}
\int_{B} \mu |f| = \int_{B} \mu |f| 
\chi_{\{ \mu \leqslant exp(\frac{\alpha |f|}{ 2 \|f\|_{BMO}})\}}
+\int_{B} \mu |f| \chi_{\{ \mu > exp( \frac{\alpha |f| }{2\|f\|_{BMO} } ) \}} .
\end{equation*}

Given $p$ be such that $1<p<\frac{5}{4}$, and let $q = \frac{p}{p-1}$ be the conjugate exponent of $p$. So, it follows from Holder's inequality that

\begin{equation*}
\begin{split}
\int_{B} \mu |f| \chi_{\{ \mu \leqslant exp(\frac{\alpha |f|}{2\|f\|_{BMO}}) \}}\\
& \leqslant \{ \int_{B} \mu \chi_{\{ \mu \leqslant exp(\frac{\alpha |f|}{2 \|f\|_{BMO} })           \} } \}^{\frac{1}{p}} 
\{ \int_{B} \mu |f|^{q} \chi_{\{ \mu \leqslant exp(\frac{\alpha |f|}{2\|f\|_{BMO}}) \}}\}^{\frac{1}{q}}
\end{split}
\end{equation*}

Since $t < exp(t)$, for all $t\in \mathbb{R}$, we have 
$ \frac{\alpha |f|}{2q \|f\|_{BMO}} < exp(\frac{ \alpha |f|}{2q \|f\|_{BMO}})$. Hence, we have

\begin{equation}\label{8}
\begin{split}
\int_{B} \mu |f| \chi_{\{ \mu \leqslant exp(\frac{\alpha |f|}{2\|f\|_{BMO}})\}}\\
& \leqslant  \frac{2q}{\alpha} \|f\|_{BMO} \{\int_{B} \mu  \}^{\frac{1}{p}}
\{ \int_{B} exp(\alpha \frac{|f|}{\|f\|_{BMO}})\}^{\frac{1}{q}}\\
& \leqslant K^{1-\frac{1}{p}}
\frac{2q}{\alpha } \|f\|_{BMO} (\int_{B} \mu )^{\frac{1}{p}} ,
\end{split}
\end{equation}

But, on the other hand, we have 

\begin{equation}\label{9}
\int_{B} \mu |f| \chi_{\{ \mu > exp(\frac{\alpha |f| }{2 \|f\|_{BMO}})\}}
\leqslant \int_{B} \frac{2}{\alpha } \|f\|_{BMO} \mu log^{+}\mu   .
\end{equation}

By combining inequalities (\ref{8}), and (\ref{9}), we conclude that

\begin{equation*}
\int_{B}\mu |f| \leqslant \frac{2p}{\alpha (p-1)} 
\{1+   K^{1-\frac{1}{p}}\}
\|f\|_{BMO} \{( \int_{B} \mu )^{\frac{1}{p}} + \int_{B} \mu log^{+} \mu \} .
\end{equation*}
\end{proof}

We are now ready to work with the term 
$\int_{T_{k-1}}^{1}|\int_{\mathbb{R}^{3}} \nabla (\frac{v_{k}}{|u|})u P_{k1}dx |ds$.

Indeed, by a simple application of the partial regularity theorem due to Caffarelli, Kohn, and Nirenberg, it can be shown that, if $B$ is a sufficiently large open ball centered at the origin of $\mathbb{R}^{3}$(we will choose $B$ to be large enough so that it will satisfy $|B|> 1$), then it follows that 

\begin{itemize}
\item $[\frac{1}{2}, 1]\times \mathbb{R}^{3}\cap  \{ v_{k} \geqslant 0\}$ is a subset of $[\frac{1}{2}, 1]\times B$, for all $k \geqslant 1$, and if $R$ is sufficiently lage.
\end{itemize}

On the other hand, since $\nabla (\frac{v_{k}}{|u|})u 
= -R(1-\frac{1}{2^{k}})F\chi_{\{v_{k}\geqslant 0\}} $. So, we have

\begin{equation*}
\begin{split}
|\int_{\mathbb{R}^{3}}\nabla (\frac{v_{k}}{|u|})u P_{k1}dx|\\
&= |\int_{B} R(1-\frac{1}{2^{k}})F\chi_{\{v_{k}\geqslant 0 \}}P_{k1}dx|\\
&\leqslant R\int_{B}|F|\chi_{\{v_{k}\geqslant 0 \}} |P_{k1} - (P_{k1})_{B}|dx\\
&+  R\int_{B}|F|\chi_{\{v_{k}\geqslant 0\}}|(P_{k1})_{B}|dx ,
\end{split} 
\end{equation*}

for all $k\geqslant 1$, and all $\frac{1}{2} < t < 1$, provided that $R$ is sufficiently large (here, the symbol $(P_{k1})_{B}$ stands for the average value of
 $P_{k1}$ over the ball $B$ ).

Now, since $P_{k1} = \sum R_{i}R_{j} \{(1-\frac{w_{k}}{|u|})u_{i}(1-\frac{w_{k}}{|u|})u_{j}\}$, it follows from the Risez's Theorem in the theory of singular integral that
$\|P_{k1}(t,\cdot )\|_{L^{2}(\mathbb{R}^{3})} 
\leqslant C_{2}R^{\beta }\|u(t,\cdot )\|_{L^{2}(\mathbb{R}^{3})}$, for all $t\in [0,1]$, in which $C_{2}$ is some constant depending only on $2$. So, we can use the Holder's inequality to carry out the following estimation

\begin{equation*}
\begin{split}
|(P_{k1})_{B}(t)|\leqslant\frac{1}{|B|}\int_{B}|P_{k1}(t,x)|dx\\ 
&\leqslant \frac{1}{|B|^{\frac{1}{2}}} \|P_{k1}(t,\cdot )\|_{L^{2}(B)}\\
&\leqslant \frac{1}{|B|^{\frac{1}{2}}}C_{2}R^{\beta }\|u(t,\cdot )\|_{L^{2}(\mathbb{R}^{3})}\\
& \leqslant C_{2} R^{\beta }\|u\|_{L^{\infty}(0,1; L^{2}(\mathbb{R}^{3}))} .
\end{split}
\end{equation*}

We remark that the last line of the above inequality holds since our open ball $B$ is sufficiently large so that $|B| > 1$. As a result, it follows that

\begin{equation}\label{10}
\begin{split}
|\int_{\mathbb{R}^{3}}\nabla (\frac{v_{k}}{|u|}) u P_{k1}dx|\\
&\leqslant R \int_{B}|F|\chi_{\{v_{k}\geqslant 0\}}|P_{k1}- (P_{k1})_{B}|dx\\
&+  C_{2}R\|u\|_{L^{\infty}(0,1;L^{2}(\mathbb{R}^{3}))} 
\int_{B}R^{\beta } |F|\chi_{\{v_{k}\geqslant 0\}}
\end{split}
\end{equation}

Indeed, the operator $R_{i}R_{j}$ is indeed a Zygmund- Calderon operator, and so $R_{i}R_{j}$ must be a bounded operator from $L^{\infty}(\mathbb{R}^{3})$ to $BMO(\mathbb{R}^{3})$. Hence we can deduce that

\begin{equation*}
\begin{split}
\|P_{k1}(t, \cdot )- (P_{k1})_{B}(t)\|_{BMO}\\
& = \|P_{k1}(t,\cdot )\|_{BMO}\\
&\leqslant C_{0} 
\|(1-\frac{w_{k}}{|u|})u_{i}(1-\frac{w_{k}}{|u|})u_{j}\|_{L^{\infty}(\mathbb{R}^{3})}\\
&\leqslant C_{0}R^{2\beta } ,
\end{split} 
\end{equation*}

for all $t\in (0,1)$, in which $C_{0}$ is some constant depending only on $\mathbb{R}^{3}$.
So, we now apply Lemma~\ref{BMO} with $\mu = |F|\chi_{\{v_{k} \geqslant 0 \}}$, and 
$f = P_{k1} - (P_{k1})_{B}$ to deduce that

\begin{equation*}
\begin{split}
\int_{B} |F| \chi_{\{v_{k}\geqslant 0 \}}|P_{k1}-(P_{k1})_{B}|dx\\
&\leqslant \frac{2pC_{0}}{\alpha (p-1)} \{1+ K^{1-\frac{1}{p}}\}\times \\
&\{(\int_{B}R^{2p\beta }|F|\chi_{\{v_{k}\geqslant 0\}})^{\frac{1}{p}}
+ \int_{B}R^{2\beta }|F|log^{+}|F|\chi_{\{v_{k}\geqslant 0 \}} \} ,
\end{split}
\end{equation*}

in which the symbol $(P_{k1})_{B}$ stands for the mean value of $P_{k1}$ over the open ball $B$.
Since we know that $\{v_{k}\geqslant 0\}$ is a subset of $\{|u|\geqslant \frac{R}{2}\}$, for all $k\geqslant 1$, so it follows from the above inequality that

\begin{equation*}
\begin{split}
\int_{B}|F|\chi_{\{  v_{k}\geqslant 0 \}}|P_{k1} - (P_{k1})_{B}|dx\\
&\leqslant \frac{2C_{0}}{\alpha} \frac{p}{p-1}4^{p\beta}
\{ 1 + K^{1-\frac{1}{p}}\}\times\\
&\{ (\int_{B}|u|^{2p\beta}|F|\chi_{\{v_{k}\geqslant 0\}})^{\frac{1}{p}}\\
&+ \int_{B} |u|^{2\beta}|F|log^{+}|F|\cdot \chi_{\{v_{k}\geqslant 0\}}\} . 
\end{split}
\end{equation*}

So, we can conclude from inequality (\ref{10}), and the above inequality that

\begin{equation}\label{11}
\begin{split}
\int_{T_{k-1}}^{1}|\int_{\mathbb{R}^{3}}\nabla (\frac{v_{k}}{|u|})u P_{k1}dx|dt\\
&\leqslant R \frac{2C_{0}}{\alpha }\frac{p}{p-1}4^{p\beta } (1+K^{1-\frac{1}{p}})\times \\ 
&\{(\int_{Q_{k-1}}|u|^{2p\beta }|F|\chi_{\{v_{k}\geqslant 0 \}})^{\frac{1}{p}}\\
& + \int_{Q_{k-1}}|u|^{2\beta }|F|log(1+|F|)\chi_{\{v_{k}\geqslant 0 \}} \}\\
& + C_{2}2^{\beta }R \|u\|_{L^{\infty}(0,1; L^{2}(\mathbb{R}^{3}))}
\int_{Q_{k-1}}|u|^{\beta }|F|\chi_{\{v_{k}\geqslant 0 \}} . 
\end{split}
\end{equation}

 Now, notice that

\begin{equation}\label{12}
\begin{split}
\int_{Q_{k-1}} |u|^{2\beta }|F|log(1+|F|)\chi_{\{v_{k}\geqslant 0\}}\\
&\leqslant \int_{Q_{k-1}} |u|^{2\beta }|F|log(1+|F|)\chi_{\{|F|\leqslant \frac{1}{R}\}}\chi_{\{v_{k}\geqslant 0   \}}\\
& + \int_{Q_{k-1}}|u|^{2\beta }|F|log(1+|F|)\chi_{\{|F|>\frac{1}{R}\}}\chi_{\{v_{k}\geqslant 0\}}\\
&\leqslant \frac{log2}{R}\int_{Q_{k-1}}|u|^{2\beta }\chi_{\{v_{k}\geqslant 0 \}}\\
&+ \int_{Q_{k-1}}|u|^{2\beta }|F|log(1+|F|)\chi_{\{|F|>\frac{1}{R}\}}
\chi_{\{v_{k}\geqslant 0\}} .
\end{split}
\end{equation}

\noindent{\bf Step five}

To deal with the second term in the last line of inequality (\ref{12}), we consider the sequence $\{\phi_{k} \}_{k=1}^{\infty}$ of nonnegative continuous functions on $[0,\infty)$, which are defined by 

\begin{itemize}
\item $\phi_{k} (t)= 0$, for all $t\in [0, C_{k}]$.
\item $\phi_{k} (t)= t - C_{k}$, for all $t\in (C_{k},  C_{k}+1)$.
\item $\phi_{k} (t)= 1$, for all $t\in [C_{k}+1, +\infty)$.
\end{itemize}

where the symbol $C_{k}$ stands for $C_{k} = R(1-\frac{1}{2^{k}})$, for every $k\geqslant 1$. Moreover, we also need a smooth function $\psi : \mathbb{R}\rightarrow \mathbb{R}$ satisfying the following conditions that:

\begin{itemize}
\item $\psi (t)=1$, for all $t\geqslant \frac{1}{R}$.
\item $0 < \psi (t) < 1$, for all $t$ with $0< t < \frac{1}{R}$.
\item $\psi (0) = 0$.
\item $-1 < \psi (t) < 0$, for all $t$ with $-\frac{1}{R} < t < 0$.
\item $\psi (t) = -1$, for all $t \leqslant -\frac{1}{R}$.
\item $0 \leqslant \frac{d}{dt}\psi  \leqslant 2R$, for all $t\in \mathbb{R}$. 
\end{itemize}

With the above preperation, let $\lambda$ be such that $2 < \lambda < \frac{10}{3} + 1 - \gamma$. We can then carry out the following calculation

\begin{equation}\label{13}
\begin{split}
div\{ |u|^{\lambda -1} u \psi (F) log(1+|F|) \phi_{k}(|u|)\}\\
&= -(\lambda -1)|u|^{\lambda}F \psi (F) log(1+|F|)\phi_{k}(|u|)\\
&-|u|^{\lambda +1}F\psi (F) log(1+|F|)\chi_{\{C_{k}\leqslant |u| \leqslant C_{k} +1\}}\\
&+ |u|^{\lambda -1} \frac{d\psi }{dt}(F) (u\cdot \nabla F)log(1+|F|)\phi_{k}(|u|)\\
& + |u|^{\lambda -1}\psi (F) \frac{u\cdot \nabla |F|}{1+|F|} \phi_{k}(|u|)  . 
\end{split}
\end{equation}

Since our weak solution $u$ on $(0,1]\times \mathbb{R}^3$ satisfies 
$\frac{|u\cdot \nabla F|}{|u|^{\gamma}}\leqslant A|F|$, it follows that

\begin{itemize}
\item $|u\cdot \nabla F|(t,x) \leqslant \frac{A}{R}|u(t,x)|^{\gamma}$, if it happens that $(t,x)$ satisfies $|F(t,x)|\leqslant \frac{1}{R}$.
\item $|\frac{u\cdot \nabla |F|}{1+ |F|}| \leqslant 
\frac{|u\cdot \nabla |F||}{|F|}= \frac{|u\cdot \nabla F|}{|F|} \leqslant
A |u|^{\gamma}$ . 
\end{itemize}

So, it follows from inequality (\ref{13}) that

\begin{equation}\label{14}
\begin{split}
\Lambda_{1} + \Lambda_{2}\\
&\leqslant \int_{Q_{k-1}}|u|^{\lambda -1}|\frac{d\psi }{dt}(F)|\cdot 
|u\cdot \nabla F|log(1+|F|)\phi_{k}(|u|)\\
&+ \int_{Q_{k-1}}|u|^{\lambda -1}|\psi (F)|\cdot |\frac{u\cdot \nabla |F|}{1+|F|}|
\phi_{k} (|u|)\\
&\leqslant \int_{Q_{k-1}}|u|^{\lambda -1}(2R)(\frac{A}{R} |u|^{\gamma})
log(1+\frac{1}{R})\phi_{k}(|u|)\\
&+\int_{Q_{k-1}}|u|^{\lambda -1}\cdot A\cdot |u|^{\gamma}\phi_{k}(|u|)\\
&\leqslant A(1+2log2)\int_{Q_{k-1}}|u|^{\lambda -1 + \gamma} \phi_{k}(|u|)\\
&\leqslant A(1+2log2)\int_{Q_{k-1}}|u|^{\lambda -1 +\gamma} \chi_{\{v_{k}\geqslant 0 \}} ,
\end{split}
\end{equation}

in which the terms $\Lambda_{1}$, and $\Lambda_{2}$ are given by

\begin{itemize}
\item $\Lambda_{1} = (\lambda -1)\int_{Q_{k-1}} |u|^{\lambda}F\psi (F)\cdot log(1+|F|)
\phi_{k}(|u|)$.
\item $\Lambda_{2} = \int_{Q_{k-1}} |u|^{\lambda +1}(\psi (F)F)\cdot log(1+|F|)
\chi_{\{C_{k}\leqslant |u| \leqslant C_{k} +1 \}}   $ .
\end{itemize}

We then notice that

\begin{itemize}
\item Since $\lambda > 2$, we have $\Lambda_{1} \geqslant
\int_{Q_{k-1}}|u|^{\lambda }(F\psi (F))log(1+|F|)\chi_{\{ |u|\geqslant C_{k} +1 \}}   $.
\item $\Lambda_{2} \geqslant \frac{R}{2}\int_{Q_{k-1}}|u|^{\lambda}F\psi(F)log(1+|F|)
\chi_{\{C_{k} \leqslant |u| \leqslant C_{k} + 1 \}} $, for every $k\geqslant 1$. Notice that this is true because $C_{k}= R(1-\frac{1}{2^k})$, and that $(1-\frac{1}{2^{k}}) \geqslant \frac{1}{2}$, for every $k\geqslant 1$.
\end{itemize}

Hence,  it follows from inequality (\ref{14}) that

\begin{equation}\label{15}
\begin{split}
\int_{Q_{k-1}}|u|^{\lambda }F\psi (F)log(1+|F|)\chi_{\{v_{k}\geqslant 0\}}\\
&= \int_{Q_{k-1}}|u|^{\lambda }F\psi (F)log(1+|F|)\chi_{\{C_{k}\leqslant |u|
\leqslant C_{k} +1 \}}\\
&+ \int_{Q_{k-1}}|u|^{\lambda }F\psi (F)log(1+|F|)\chi_{\{|u|> C_{k} +1\}}\\
&\leqslant \frac{2}{R}\Lambda_{2} + \Lambda_{1}\\
&\leqslant 3C\cdot A\int_{Q_{k-1}}|u|^{\lambda -1 +\gamma }\chi_{\{v_{k} \geqslant 0\}}. 
\end{split}
\end{equation}

As a matter of fact, inequality (\ref{15}) leads us to raise up the index for the term $\int_{Q_{k-1}} |u|^{\theta} \chi_{\{v_{k}\geqslant 0\}}$, for any $\theta$ with $ 0  < \theta < \frac{10}{3}$, in the following way

\begin{equation*}
\begin{split}
\int_{Q_{k-1}}|u|^{\theta }\chi_{\{v_{k}\geqslant 0 \}}\\
& = \int_{Q_{k-1}} \{R(1-\frac{1}{2^{k}}) + v_{k} \}^{\theta } \chi_{\{v_{k}\geqslant 0 \}}\\
&\leqslant C_{\theta} \{R^{\theta }\int_{Q_{k-1}}\chi_{\{v_{k}\geqslant 0 \}} +
\int_{Q_{k-1}} v_{k}^{\theta }\chi_{\{v_{k} \geqslant 0 \}} \}\\
&\leqslant \frac{C_{\theta }}{R^{\frac{10}{3}-\theta }} \{2^{\frac{10k}{3}}
+ 2^{(\frac{10}{3} - \theta  )k}\}\int_{Q_{k-1}}v_{k-1}^{\frac{10}{3}}\\
&\leqslant \frac{C_{\theta }}{R^{\frac{10}{3}-\theta }}2^{\frac{10k}{3}}U_{k-1}^{\frac{5}{3}} .  
\end{split}
\end{equation*}

for every $\theta$ with $0< \theta < \frac{10}{3}$, where $C_{\theta }$ is some positive constant depending only on $\theta$. Hence it follows from inequalities(\ref{12}), (\ref{15}), and our last inequality that

\begin{equation}\label{16}
\begin{split}
\int_{Q_{k-1}} |u|^{2\beta } |F|\cdot log(1+|F|)\chi_{\{v_{k}\geqslant 0  \}}\\
&\leqslant \frac{log2}{R}\int_{Q_{k-1}}|u|^{2\beta }\chi_{\{v_{k}\geqslant 0\}}\\
&+ \int_{Q_{k-1}} |u|^{2\beta }|F|log(1+|F|)\chi_{\{|F|>\frac{1}{R}\}}
\chi_{\{v_{k}\geqslant 0 \}}\\
&\leqslant \frac{log2}{R} \frac{C_{2\beta }2^{\frac{10k}{3}}}{R^{\frac{10}{3}-2\beta }}U_{k-1}^{\frac{5}{3}}\\
&+ 3C\cdot A \int_{Q_{k-1}}|u|^{2\beta -1 +\gamma }\chi_{\{v_{k}\geqslant 0\}}\\
&\leqslant C_{\beta , \gamma }(1+A)\cdot 2^{\frac{10k}{3}}U_{k-1}^{\frac{5}{3}}
\{ \frac{1}{R^{\frac{10}{3}-2\beta +1}}\\
& + \frac{1}{R^{\frac{10}{3}-2\beta +1-\gamma    }}  \} ,
\end{split}
\end{equation}

in which $\beta > \frac{3}{2}$, and that $\beta$ is sufficiently close to $\frac{3}{2}$, and $C_{\beta , \gamma }$ is some constant depending only on $\beta$, and$\gamma$.

Before we can finish our job, we also need to deal with the term 
$(\int_{Q_{k-1}}|u|^{2p\beta }|F|\chi_{\{v_{k}\geqslant 0 \}} )^{\frac{1}{p}}$,
and the term $\int_{Q_{k-1}}|u|^{\beta }|F|\chi_{\{v_{k}\geqslant 0\}}$,
 which appear in inequality (\ref{11}). For this purpose, we will consider $\lambda$ again to be $\frac{3}{2} < \lambda < \frac{10}{3}+1- \gamma $, and let us carry out the following computation, in which $\psi$ and $\phi_{k}$ etc are just the same as before. 

\begin{equation*}
\begin{split}
div \{|u|^{\lambda -1 }u\psi (F) \phi_{k}(|u|)\}\\
&= -(\lambda -1)|u|^{\lambda }F\psi (F)\phi_{k}(|u|)\\
& + |u|^{\lambda -1}\frac{d\psi }{dt}(F) (u\cdot \nabla F )\phi_{k}(|u|)\\
&- |u|^{\lambda +1}F\psi (F)\chi_{\{C_{k}\leqslant |u| \leqslant C_{k} +1   \}} .
\end{split}
\end{equation*}

Hence, we have

\begin{equation*}
\begin{split}
(\lambda -1)\int_{Q_{k-1}}|u|^{\lambda}F\psi (F)\phi_{k} (|u|)\\
& +\int_{Q_{k-1}}|u|^{\lambda +1}F\psi (F)\chi_{\{C_{k} \leqslant |u| \leqslant C_{k} +1 \}}\\
&\leqslant \int_{Q_{k-1}} |u|^{\lambda -1}|\frac{d\psi }{dt}(F)|\cdot 
|u\cdot \nabla F|\phi_{k}(|u|)\\
&\leqslant \int_{Q_{k-1}} |u|^{\lambda -1} (2R)(\frac{A}{R}|u|^{\gamma })
\chi_{\{v_{k}\geqslant 0\}}\\
&\leqslant 2A\int_{Q_{k-1}} |u|^{\lambda -1+ \gamma }\chi_{\{v_{k}\geqslant 0 \}} .
\end{split}
\end{equation*}

So, it follows that

\begin{equation*}
\begin{split}
\int_{Q_{k-1}}|u|^{\lambda } F\psi (F) \chi_{\{v_{k}\geqslant 0 \}}\\
& = \int_{Q_{k-1}} |u|^{\lambda }F\psi (F)\chi_{\{C_{k}\leqslant |u| \leqslant C_{k} + 1   \}}\\
&+ \int_{Q_{k-1}}|u|^{\lambda } F\psi (F)\chi_{\{|u|> C_{k} + 1  \}}\\
&\leqslant \frac{2}{R}\int_{Q_{k-1}}|u|^{\lambda + 1} F\psi (F)
\chi_{\{C_{k} \leqslant |u| \leqslant C_{k} + 1 \}}\\
& + \int_{Q_{k-1}}|u|^{\lambda }F\psi (F) \phi_{k}(|u|)\\
&\leqslant 3 \{\int_{Q_{k-1}}|u|^{\lambda +1}F\psi (F)\chi_{\{C_{k}\leqslant |u| \leqslant C_{k} +1 \}}\\    
& +(\lambda -1)\int_{Q_{k-1}}|u|^{\lambda }F\psi(F)\phi_{k}(|u|)\}\\
&\leqslant 6A \int_{Q_{k-1}}|u|^{\lambda -1 +\gamma }\chi_{\{v_{k}\geqslant 0 \}},
\end{split}
\end{equation*}

in which $\lambda$ satisfies $\frac{3}{2} < \lambda < \frac{10}{3} + 1 - \gamma $. Now, put $\lambda = 2p\beta $, with $\beta > \frac{3}{2}$ to be sufficiently close to 
$\frac{3}{2}$, and $1<p<\frac{5}{4}$ to be sufficiently close to $1$, it follows from our last inequality that

\begin{equation}\label{17}
\begin{split}
\int_{Q_{k-1}}|u|^{2p\beta }|F|\chi_{\{v_{k}\geqslant 0 \}}\\
& = \int_{Q_{k-1}}|u|^{2p\beta }|F|\chi_{\{|F|\leqslant \frac{1}{R}\}}
\chi_{\{v_{k}\geqslant 0\}}\\
&  +\int_{Q_{k-1}} |u|^{2p\beta }\chi_{\{|F|> \frac{1}{R}\}}
\chi_{\{v_{k}\geqslant 0  \}}|F|\\
&\leqslant \frac{1}{R}\int_{Q_{k-1}}|u|^{2p\beta }\chi_{\{v_{k}\geqslant 0\}}\\
&+ 6A\int_{Q_{k-1}}|u|^{2p\beta -1 + \gamma }\chi_{\{v_{k}\geqslant 0 \}}\\
&\leqslant C(1+A)\{\frac{1}{R^{\frac{10}{3}-2p\beta +1}} 
+\frac{1}{R^{\frac{10}{3}-2p\beta +1 -\gamma }} \}2^{\frac{10k}{3}}U_{k-1}^{\frac{5}{3}} .
\end{split}
\end{equation}

In exactly the same way, by setting $\lambda$ to be $\beta $, with 
$\beta > \frac{3}{2}$ to be sufficiently close to $\frac{3}{2}$, it also follows that

\begin{equation}\label{18}
\begin{split}
\int_{Q_{k-1}}|u|^{\beta } |F|\chi_{\{v_{k}\geqslant 0 \}}\\
&= \int_{Q_{k-1}}|u|^{\beta }|F|\chi_{\{|F|\leqslant \frac{1}{R}\}}\chi_{\{v_{k}\geqslant 0\}}\\
&+\int_{Q_{k-1}}|u|^{\beta }|F|\chi_{\{|F|>\frac{1}{R}\}}\chi_{\{v_{k}\geqslant 0 \}}\\
&\leqslant \frac{1}{R}\int_{Q_{k-1}}|u|^{\beta }\chi_{\{v_{k}\geqslant 0 \}}
+6A\int_{Q_{k-1}}|u|^{\beta -1 +\gamma }\chi_{\{v_{k}\geqslant 0 \}}\\
&\leqslant C_{\beta ,\gamma }(1+ A) \{\frac{1}{R^{\frac{10}{3}-\beta +1}} 
+ \frac{1}{R^{\frac{10}{3}-\beta +1-\gamma }}\} 2^{\frac{10k}{3}}U_{k-1}^{\frac{5}{3}} .
\end{split}
\end{equation}

By combining inequalities (\ref{11}), (\ref{16}), and (\ref{17}),and (\ref{18}) we now conclude that

\begin{equation}\label{19}
\begin{split}
\int_{Q_{k-1}}|\int_{Q_{k-1}} \nabla (\frac{v_{k}}{|u|})uP_{k1}dx|ds\\
&\leqslant (1+A)(1+\frac{1}{\alpha })C_{p, \beta }(1+K^{1-\frac{1}{p}})\times \\ &(1+\|u\|_{L^{\infty}(0,1;L^{2}(\mathbb{R}^{3}))})\times \\
&\{  (\frac{1}{R^{\frac{10}{3}-2p\beta +1-\gamma - p}})^{\frac{1}{p}} 2^{\frac{10k}{3p}} U_{k-1}^{\frac{5}{3p}}   
 + \frac{1}{R^{\frac{10}{3}-2\beta -\gamma }} 2^{\frac{10k}{3}}U_{k-1}^{\frac{5}{3}}
 \} .
\end{split}
\end{equation}

Notice that if $p\rightarrow 1^{+}$, and $\beta \rightarrow \frac{3}{2}^{+}$, then,we have $(\frac{10}{3}-2p\beta +1 -p -\gamma)\rightarrow (\frac{1}{3}-\gamma ) >0$, and that $(\frac{10}{3}-2\beta -\gamma )\rightarrow (\frac{1}{3} - \gamma )> 0$.\\

So, finally, we recognize that by combining inequalities (\ref{middle}), (\ref{easy}), and (\ref{19}), we conclude that we are done in proving proposition~\ref{index} .

\end{document}